\documentclass[10pt]{article}     % -*- latex -*-
\usepackage{amsthm}
\usepackage{amsfonts, amssymb, amsmath, amscd, latexsym}
\usepackage{mathrsfs}
\usepackage{epsfig}
\usepackage[latin1]{inputenc}
\usepackage[all]{xy}

\newtheorem{theorem}{Theorem}[section]

\newtheorem{proposition}[theorem]{Proposition}
\newtheorem{remark}[theorem]{Remark}

\newtheorem{assumption}[theorem]{Assumption}

\numberwithin{equation}{section}

\def\ds{\displaystyle}
\def\N{\mathbb {N}}

\def\R{\mathbb {R}}
\newcommand\Inter[1]{{\stackrel{\circ}{#1}}}

\newcommand\pscal[2]{\left\langle \, #1  \mid  #2 \, \right\rangle}	% Prodotto scalare, uso: \pscal{x}{y}
\newcommand{\vleq}{\rotatebox[origin=c]{-90}{$\leq$}}

\date{}
\title{Double resonance in Sturm-Liouville planar boundary value problems}
\author{Andrea Sfecci
\thanks{Dipartimento di Scienze Matematiche,
Universit\`a Politecnica delle Marche, Via Brecce Bianche 1, 60131 Ancona -
Italy. email: sfecci@dipmat.univpm.it }}

\begin{document}

\maketitle

\begin{abstract}
We provide some existence results for Sturm-Liouville boundary value problems associated with the planar differential system $Jz'=g(t,z) + r(t,z)$ where $g$ is suitably controlled by the gradient of two positively homogeneous functions of degree 2 and $r$ is bounded. We study the existence of solutions when a double resonance phenomenon occurs by the introduction of Landesman-Lazer type of conditions. Applications to scalar second order differential equations are given.
\end{abstract}



\section{Introduction}\label{intro}

For the scalar equation
\begin{equation}\label{equno}%$$
x''+f(t,x)=0
\end{equation}%$$
with perdiodic, Neumann or Dirichlet boundary conditions
there have been a lot of researches concerning the existence of solutions under some nonresonance conditions.

The approach to resonance is a delicate problem and the most successful condition have been introduced by Landesman and Lazer, when the nonlinearity asymptotically lies between two eigenvalues of the linear differential equation, see e.g.~\cite{FaFo1990,FaFo1992}.
In the case of asymmetric nonlinearities we mention~\cite{L4} for the periodic case and~\cite{FoGar2011} for planar systems (see also the monograph~\cite{Fondabook} for further informations on this topic).

Even though double resonance phenomenon has been studied dealing with periodic boundary conditions, such a discussion has not been treated for Dirichlet and Neumann boundary conditions yet. In this paper we are going to present some existence results at double resonance for equation
\eqref{equno} 
when $f$ satisfies
\begin{equation}\label{asym1}
0<\nu_1\leq\liminf_{x\to-\infty} \frac{f(t,x)}{x}\leq\limsup_{x\to-\infty} \frac{f(t,x)}{x} \leq \nu_2\,,
\end{equation}
\begin{equation}\label{asym2}
0<\mu_1 \leq\liminf_{x\to+\infty} \frac{f(t,x)}{x}\leq\limsup_{x\to+\infty} \frac{f(t,x)}{x} \leq \mu_2
\end{equation}
(see Theorem~\ref{TD} below).

In such a situation the nonlinearity $f$ ``mimes'' an asymmetric oscillator
$$x''+\mu x^+- \nu x^-=0\,,$$
where $x^+=(|x|+x)/2$ and $x^-=(|x|-x)/2$. The previous scalar differential equation can be studied as a planar system of the type
\begin{equation}\label{plan-auto}
Jz' = \nabla V(z)\,, \quad z\in\R^2\,,
\end{equation}
where  $
J={\left(
\begin{array}{cc} 0&-1\\1&0 \end{array}
\right)}
$
is the standard symplectic matrix
and
$V:\R^2\to\R$ is a positively homogeneous $C^1$-function of degree 2, i.e. such that
$$%\begin{equation}\label{poshom2}
0<V(\lambda z )= \lambda^2 V(z)\,, \qquad \text{for every } \lambda>0\,, \, z\neq 0\,.
$$%\end{equation}
For this reason, boundary value problems related to~\eqref{plan-auto} present a particular interest in literature, see e.g.~\cite{BosGar,Fonda2004,FoGar2013,Tom} and the references therein.

In relation with the scalar second order differential equation \eqref{equno},
the Dirichlet boundary conditions $x(0)=x(T)=0$ (DBC), 
the Neumann boundary conditions $x'(0)=x'(T)=0$ (NBC)
and the mixed boundary conditions $x'(0)=x(T)=0$ (MBC)
can be collected all together in a unique class of problems when we pass to consider planar systems as in~\eqref{plan-auto}. % or~\eqref{plan-mine}. 
Indeed, we can ask a solution $z(t)=(x(t),y(t))$ to start and arrive at some points belonging to two lines in the plane:
\begin{equation}\label{boundary-i}
z(0)\in l_S\,, \qquad 
z(T)\in l_A\,,
\end{equation}
where $l_S$ is the {\sl starting line} and $l_A$ is the {\sl arrival line}. 
In particular 
(DBC) is equivalent to the case $l_S=l_A=\{z=(x,y)\mid x=0\}$,
(NBC) is equivalent to the case $l_S=l_A=\{z=(x,y)\mid y=0\}$ and
(MBC) is equivalent to the case $l_S=\{z=(x,y)\mid y=0\}$ and $l_A=\{z=(x,y)\mid x=0\}$.

\medbreak

In~\cite{BosGar,FoGar2013} the following class of problems, obtained as a perturbation of~\eqref{plan-auto}, is treated: %Boscaggin, Fonda and Garrione studied the case $V_1=V_2$, i.e. the problem
\begin{equation}\label{A1}
\begin{cases}
J z' = \nabla V(z) + p(t,z)\,,\\
z(0)\in l_S\,, \qquad 
z(T)\in l_A\,,
\end{cases}
\end{equation}
where, for briefness, we say $p$ is bounded and continuous.
%
%The study of existence of solutions to problem~\eqref{A1} is related to the study of perturbed asymmetric oscillators, e.g. differential equations as $x''+\mu x^+- \nu x^-+g(x)=e(t)$ where $g$ and $e$ are bounded continuous functions and further generalizations.
%%with prescribed boundary conditions 
%Such a type of problems presents a wide literature, see e.g.~\cite{L1a,L1b,L2,L3,L4,L5,Fonda2004,L6,L7,L8,L9} for a nonexhaustive bibliography, apologizing for unavoidable missing references.

\medbreak

Recalling that the unperturbed system~\eqref{plan-auto}
has an isochronous center of minimal period $\tau_V$, and borrowing the definition from~\cite{BosGar}, we say that the unperturbed problem
\begin{equation}\label{A2}
\begin{cases}
J z' = \nabla V(z)\,,\\
z(0)\in l_S\,, \qquad 
z(T)\in l_A\,.
\end{cases}
\end{equation}
is {\sl resonant} if it has at least one nontrivial solution.
As in the periodic case, if problem~\eqref{A2} is not resonant then a perturbed problem as in~\eqref{A1} admit a solution, cf.~\cite{FoGar2013}. Conversely, if the unperturbed problem~\eqref{A2} is resonant, then the existence of a solution to problems as in~\eqref{A1} is ensured assuming an additional condition: in~\cite{BosGar} the introduction of a Landesman-Lazer type of assumption provides an existence result. 
%In~\cite{FoGar2013}, Fonda and Garrione studied problems~\eqref{A1} such that the unperturbed problem~\eqref{A2} is not resonant. Conversely, in~\cite{BosGar}, Boscaggin and Garrione studied problems~\eqref{A1} such that the unperturbed problem~\eqref{A2} is resonant. In the latter case the existence of solutions is provided by the introduction of some Landesman-Lazer type of Assumptions.
In these notes we continue the study performed in~\cite{BosGar,FoGar2013} by Boscaggin, Fonda and Garrione.
In particular, we are going to consider the wider class of problems
\begin{equation}\label{plan-mine}
\begin{cases}
J z' = g(t,z) + p(t,z)\,,\\
z(0)\in l_S\,, \qquad 
z(T)\in l_A\,,
\end{cases}
\end{equation}
where the function $g$ is controlled by two positively homogeneous functions of degree 2, $V_1\leq V_2$: More precisely, $g$ satisfies
$$g(t,z)=\big(1-\gamma(t,z)\big) \nabla V_1(z) + \gamma(t,z) \nabla V_2(z)\,,$$
where
$\gamma:[0,T]\times \R^2\to\R$, with $0\leq\gamma\leq 1$
 and
$p:[0,T]\times\R^2\to\R^2$ is sublinear with respect to the second variable.

The study of existence of solutions to problem~\eqref{plan-mine} is related to the study of perturbed asymmetric oscillators, e.g. differential equations as $x''+\mu x^+- \nu x^-+g(x)=e(t)$ where $g$ and $e$ are bounded continuous functions. In particular, \eqref{plan-mine} includes the previously mentioned scalar differential equation \eqref{equno} with $f$ obeying to~\eqref{asym1} and~\eqref{asym2}, as a particular case.
%with prescribed boundary conditions 
Such a type of problems presents a wide literature, see e.g.~\cite{L1a,L1b,L2,L3,L4,L5,Fonda2004,Fondabook,L6,L7,L8,L9} for a nonexhaustive bibliography, apologizing for unavoidable missing references.

% and $|p(t,z)|\leq \ell(t)$, for a.e. $t\in[0,T]$ and $z\in\R^2$, for a suitable $\ell \in L^2(0,T)$.

\medbreak

Systems as in~\eqref{plan-mine} have been investigated in~\cite{FoGar2011} by Fonda and Garrione dealing with periodic boundary conditions (see also~\cite{FaFo2005,FoMaw2006}).
In the periodic setting, the existence of solutions can be ensured if there exists a positive integer $k$ such that
\begin{equation}\label{tauper}
\tfrac{T}{k+1} \leq \tau_{V_2} \leq \tau_{V_1} \leq \tfrac{T}{k}\,,
\end{equation}
where $\tau_{V_1}$ and $\tau_{V_2}$ are the periods of the solutions of system~\eqref{plan-auto} choosing respectively $V=V_1$ and $V=V_2$.

In such a situation we can distinguish three situations: nonresonance, when we have the strict inequalities in~\eqref{tauper}; simple resonance, when we have a strict inequality and an equality in~\eqref{tauper}; double resonance, when both equalities hold in~\eqref{tauper}.
In presence of resonance we need to add additional assumptions, e.g. of Landesman-Lazer type, as suggested in~\cite{FoGar2011} (see also~\cite{L4,FaFo1990,FaFo1992,S-ampa,S-ans} for related results).

\medbreak

In this paper, dealing with system~\eqref{plan-mine}, we investigate all the three situations: nonresonance, simple resonance and double resonance, which will be treated in Section~\ref{mainresults}. The situations differ depending on the position of the value $T$ in~\eqref{plan-mine} with respect to a {\sl resonance set} which will be introduced in~\eqref{union}.

\medbreak

The paper is organized as follows.
In Section~\ref{prel} we present some preliminary results: in Section~\ref{auto} some properties of the autonomous system~\eqref{plan-auto} are listed borrowing some notations from~\cite{BosGar,FoGar2013}, then in Section~\ref{pert} we add a first perturbation presenting some properties of solutions of systems~\eqref{plan-mine} in the the {\sl semi-autonomous} case $r\equiv 0$. We present the main Theorems~\ref{thm-nonres} (nonresonance), \ref{thmLL} (simple resonance) and~\ref{thmLL2} (double resonance) in the successive Sections~\ref{secnonres}, \ref{secsimres} and~\ref{secdoubleres}.
Finally, in Section~\ref{secapp}, we present the applications of our theorems in the case of scalar equations \eqref{equno} with an asymmetric nonlinearity, cf. Theorem~\ref{TD}.

\medbreak

Let us here introduce some notations. We will denote by $|\cdot|$ the Euclidean norm in $\R^2$ and we will use the complex notation for polar coordinates in the plane, i.e. $z=(x,y)=\rho e^{i\vartheta}=(\rho\cos\vartheta,\rho\sin\vartheta)$. Moreover, in order to well define the angle $\vartheta$ when we pass to polar coordinates, we will consider functions $z:I\to\R^2$ such that $z(t)\neq(0,0)$ for every $t\in I$. For briefness we call them {\em never-zero} functions.

\section{Preliminaries}\label{prel}

\subsection{An autonomous isochronous planar system}\label{auto}

In this section we recall some notations and contents from~\cite{BosGar,FoGar2013}.
Let us consider the planar system
\begin{equation}\label{iso}
J z' = \nabla V(z)\,, \qquad z=(x,y)\in\R^2\,,
\end{equation}
where  $
J={\left(
\begin{array}{cc} 0&-1\\1&0 \end{array}
\right)}
$
is the standard symplectic matrix and $V:\R^2\to\R$ is a $C^1$-function  which is positively homogeneous of degree $2$, i.e.
\begin{equation}\label{poshom2}
0<V(\lambda z )= \lambda^2 V(z)\,, \qquad \text{for every } \lambda>0\,, z\neq 0\,.
\end{equation}
Let us recall the validity of the Euler's formula: $\pscal{\nabla V(z)}{z} = 2V(z)$ for every $z\in\R^2$.

System~\eqref{iso} is an isochronous center of minimal period
\begin{equation}\label{period}
\tau_V = \int_0^{2\pi} \frac{d\theta}{2V(\cos\theta,\sin\theta)}
\end{equation}
and all the solutions have the form $z(t) = C \varphi_V(t+\tau)$, with $C\geq 0$ and $\tau\in[0,\tau_V)$, where $\varphi_V$ is a fixed nontrivial solution to~\eqref{iso}. Without loss of generality we assume $V(\varphi_V(t))\equiv \frac 12$ and $\varphi_V(0)=(0,y_0)$ with $y_0>0$.

Let us consider the following boundary condition
\begin{equation}\label{boundary}
z(0)\in l_S\,, \qquad 
z(T)\in l_A\,,
\end{equation}
where $l_S$ and $l_A$ (``$S$'' stands for {\sl starting}, ``$A$'' for {\sl arrival}) are lines through the origin of slope $\zeta_S$ and $\zeta_A$, respectively. We mean that a line through the origin has slope 
$\zeta\in(-\pi/2,\pi/2]$ if it can be parametrized as $l:\R\to\R^2$, $l(s)=s(\cos\zeta,\sin\zeta)$.

For later purpose, let us introduce
\begin{equation}\label{deltazeta}
\Delta\zeta = \begin{cases}
\zeta_S-\zeta_A & \text{if } \zeta_S > \zeta_A\\
\zeta_S-\zeta_A + \pi & \text{if } \zeta_S \leq \zeta_A\,,
\end{cases}
\end{equation}
which is the smallest positive angle a solution covers moving from $l_S$ to $l_A$, cf. Figure~\ref{fig:tau} (remember that solutions rotate clockwise).

Let us now recall some notations:
\begin{itemize}
\item $\tau_{0,V}$ is the least nonnegative time such that $\varphi_V(\tau_{0,V})\in l_S$,
\item $\tau_{1,V}$ is the least positive time such that $\varphi_V(\tau_{0,V}+\tau_{1,V})\in l_A$,
\item $\sigma_{1,V}$ is the least nonnegative time such that $\varphi_V(\tau_{0,V}+\tau_{1,V}+\sigma_{1,V})\in l_S$,
\item $\tau_{2,V}$ is the least positive time such that $\varphi_V(\tau_{0,V}+\tau_{1,V}+\sigma_{1,V}+\tau_{2,V})\in l_A$,
\item $\sigma_{2,V}$ is the least nonnegative time such that $\varphi_V(\tau_{0,V}+\tau_{1,V}+\sigma_{1,V}+\tau_{2,V}+\sigma_{2,V})\in l_S$.
\end{itemize}
Notice that, by definition,
\begin{equation}\label{taulaps}
\tau_V = \tau_{1,V}+\sigma_{1,V}+\tau_{2,V}+\sigma_{2,V}
\end{equation}
and in particular, as in~\eqref{period}, we have
\begin{equation}\label{starstar}
\begin{array}{cc}
\ds
\tau_{1,V}= \int_{\zeta_S-\Delta\zeta}^{\zeta_S} \frac{d\theta}{2V(\cos\theta,\sin\theta)}\,,&\ds
\tau_{2,V}= \int_{\zeta_S+\pi-\Delta\zeta}^{\zeta_S+\pi} \frac{d\theta}{2V(\cos\theta,\sin\theta)}\,,\\ \\
\ds
\sigma_{1,V}= \int^{\zeta_S-\Delta\zeta}_{\zeta_S-\pi} \frac{d\theta}{2V(\cos\theta,\sin\theta)}\,,&\ds
\sigma_{2,V}= \int^{\zeta_S+\pi-\Delta\zeta}_{\zeta_S} \frac{d\theta}{2V(\cos\theta,\sin\theta)}\,,
\end{array}
\end{equation}
As a consequence, if $l_S$ and $l_A$ coincide, then $\sigma_{1,V}=\sigma_{2,V}=0$.

In order to distinguish the two semi-lines which $l_A$ and $l_S$ consist of respectively, let us introduce the following notations (cf. Figure~\ref{fig:tau})
\begin{equation}\label{semilines}
\begin{array}{ll}
l_S^1 \ni \varphi_V(\tau_{0,V}) \,, &
l_A^1 \ni \varphi_V(\tau_{0,V}+\tau_{1,V})\,,\\[2mm]
l_S^2 \ni \varphi_V(\tau_{0,V}+\tau_{1,V}+\sigma_{1,V}) \,, &
l_A^2 \ni \varphi_V(\tau_{0,V}+\tau_{1,V}+\sigma_{1,V}+\tau_{2,V})\,.
\end{array}
\end{equation}

\begin{figure}[t]
\centerline{\epsfig{file=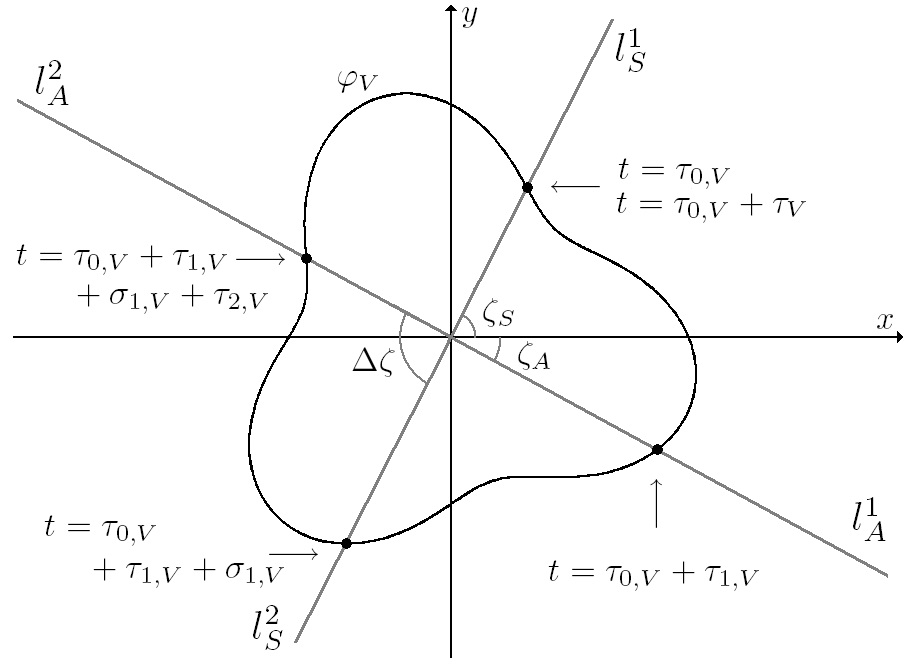, width = 10 cm}}
\caption{The notation of Section~\ref{auto}, a sketch in the case $\zeta_A < \zeta_S$.
%The lines $\ell^S$ and $\ell^A$ and the values introduced in~\eqref{deltazeta},~\eqref{taulaps},~\eqref{semilines}.
}
\label{fig:tau}
\end{figure}

The problem
\begin{equation}\label{problem}
\begin{cases}
J z' = \nabla V(z)\,,\\
z(0)\in l_S\,, \qquad 
z(T)\in l_A\,,
\end{cases}
\end{equation}
is said to be {\sl resonant} if it admits nontrivial solutions. Such a situation occurs if and only if $T$ satisfies at least one of the following identities, for a certain $k\in\N$,
\begin{eqnarray}
T&=& k\tau_V + \tau_{1,V}\,,														 	\label{resonancesetV1}\\
T&=& k\tau_V + \tau_{1,V}+\sigma_{1,V}+\tau_{2,V}\,, 			\label{resonancesetV2}\\
T&=& k\tau_V + \tau_{2,V}\,, 															\label{resonancesetV3}\\
T&=& k\tau_V + \tau_{2,V}+\sigma_{2,V}+\tau_{1,V}\,. 			\label{resonancesetV4}
\end{eqnarray}
and a nontrivial solution is given by
\begin{eqnarray}
\varphi_V(t+\tau_{0,V}) && \text{if ~\eqref{resonancesetV1} or~\eqref{resonancesetV2} holds,} \label{eigen12}\\
\varphi_V(t+\tau_{0,V}+\tau_{1,V}+\sigma_{1,V}) && \text{if ~\eqref{resonancesetV3} or~\eqref{resonancesetV4} holds.}\label{eigen34}
\end{eqnarray}
Indeed, we can distinguish different starting semilines and arrival semilines:
\begin{eqnarray*}
\text{\eqref{resonancesetV1} holds} \Rightarrow \text{the solution starts from } l_S^1 \text{ and arrives on } l_A^1\,,\\
\text{\eqref{resonancesetV2} holds} \Rightarrow \text{the solution starts from } l_S^1 \text{ and arrives on } l_A^2\,,\\
\text{\eqref{resonancesetV3} holds} \Rightarrow \text{the solution starts from } l_S^2 \text{ and arrives on } l_A^2\,,\\
\text{\eqref{resonancesetV4} holds} \Rightarrow \text{the solution starts from } l_S^2 \text{ and arrives on } l_A^1\,.
\end{eqnarray*}

\subsection{Introducing a perturbation in the energy}\label{pert}

We now focus our attention on the qualitative properties of solutions to the boundary value problem
\begin{equation}\label{system}
\begin{cases}
J z' = g(t,z)\,,\\
z(0)\in l_S\,, \qquad 
z(T)\in l_A\,,
\end{cases}
\end{equation}
where the function $g:[0,T]\times \R^2 \to \R^2$ is suitably controlled by two Hamiltonians $V_1$ and $V_2$ as in the previous section. More precisely, we introduce the following assumption.

\begin{assumption}\label{ass1}
Assume that there exists a $L^2$-Carath\'eodory function $\gamma:[0,T]\times \R^2\to[0,1]$ such that
$$
g(t,z)=(1-\gamma(t,z)) \nabla V_1(z) + \gamma(t,z) \nabla V_2(z)
$$
where $V_1\leq V_2$ are two positively homogeneous $C^1$-functions as in~\eqref{poshom2}.
\end{assumption}

Let us now borrow the notations of the previous section, cf.~\eqref{taulaps}, and we define the values
\begin{eqnarray}
\begin{CD}
\tau_{V_2} &=& \tau_{1,V_2} &+& \sigma_{1,V_2} &+&\tau_{2,V_2} &+&\sigma_{2,V_2}\,, \\
\vleq &\qquad& \vleq &\qquad& \vleq &\qquad& \vleq &\qquad& \vleq \\
\tau_{V_1} &=& \tau_{1,V_1} &+& \sigma_{1,V_1} &+&\tau_{2,V_1} &+&\sigma_{2,V_1}\,. \\
\end{CD}
\end{eqnarray}

Let us set
\begin{equation}
\begin{array}{ll}
\alpha_{2k} = \min\{a_{2k}^1\,,\, a_{2k}^2 \}\,,
& a_{2k}^1= k\tau_{V_2} +\tau_{1,V_2}\,, \\[1mm]
& a_{2k}^2= k\tau_{V_2} + \tau_{2,V_2}\,; \\[2mm]
\beta_{2k}= \max\{b_{2k}^1\,,\, b_{2k}^2 \}\,, 
& b_{2k}^1= k\tau_{V_1} +\tau_{1,V_1}\,, \\[1mm]
& b_{2k}^2= k\tau_{V_1} + \tau_{2,V_1}\,; \\[2mm]
\alpha_{2k+1} = \min\{a_{2k+1}^1\,,\, a_{2k+1}^2 \}\,,
& a_{2k+1}^1= k\tau_{V_2} +\tau_{1,V_2}+\tau_{2,V_2} + \sigma_{1,V_2}\,, \\[1mm]
& a_{2k+1}^2= k\tau_{V_2} +\tau_{1,V_2}+\tau_{2,V_2} + \sigma_{2,V_2}\,; \\[2mm]
\beta_{2k+1} = \max\{b_{2k+1}^1\,,\, b_{2k+1}^2 \}\,,
& b_{2k+1}^1= k\tau_{V_1} +\tau_{1,V_1}+\tau_{2,V_1} + \sigma_{1,V_1}\,, \\[1mm]
& b_{2k+1}^2= k\tau_{V_1} +\tau_{1,V_1}+\tau_{2,V_1} + \sigma_{2,V_1}\,; \\
\end{array}
\end{equation}
and define the intervals 
$I_j= [\alpha_j \,,\, \beta_j]$, $j\in \N$.
Notice that all the intervals are well ordered in the following sense:
$\alpha_j < \alpha_{j+1}$ and $\beta_j < \beta_{j+1}$ for every $j\in\N$,

We introduce the {\sl resonance set}
\begin{equation}\label{union}
\mathcal I= \bigcup_{j\in\N} I_j = \bigcup_{j\in\N} [\alpha_j \,,\, \beta_j] \,, 
\end{equation}
the interior of $\mathcal I$, denoted by $\Inter{\mathcal I}$, and
\begin{equation}\label{uniontilde}
\widetilde{\mathcal I} = \bigcup_{j\in\N} (\alpha_j \,,\, \beta_j) \,.
\end{equation}
Notice that $\widetilde{\mathcal I}\subseteq \Inter{\mathcal I}\subseteq \mathcal I$ and so $\partial \mathcal I\subseteq \partial\widetilde{\mathcal I}$.

In this paper we are going to treat the following situations:
\begin{itemize}
\item {\sl Nonresonance}: $T\notin \mathcal I$, i.e. $\exists \kappa\in\N$ such that $\beta_\kappa<T<\alpha_{\kappa+1}$, or $T<\alpha_0$.
\item {\sl Simple resonance}: $T\in \partial \mathcal I$, i.e. $\exists \kappa\in\N$ such that $T=\alpha_\kappa$ or $T=\beta_\kappa$.
\item {\sl Double resonance}: $T\in \partial\widetilde{\mathcal I}\setminus \partial \mathcal I$, i.e. $\exists \kappa\in\N$ such that $T=\beta_\kappa =\alpha_{\kappa+1}$.
\end{itemize}

\medbreak

Introducing polar coordinates $z=(x,y)=\rho e^{i\vartheta}$, the angular velocity of a {\em never-zero} solution of~\eqref{system} is given by
\begin{multline*}
-\vartheta'(t)= \frac{\pscal{Jz'(t)}{z(t)}}{|z(t)|^2} \\
=(1-\gamma(t,z(t)) \frac{\pscal{\nabla V_1(z(t))}{z(t)}}{|z(t)|^2} + 
\gamma(t,z(t)) \frac{\pscal{\nabla V_2(z(t))}{z(t)}}{|z(t)|^2} \\
= 2(1-\gamma(t,z(t)) V_1(\cos\vartheta(t),\sin\vartheta(t)) +
2\gamma(t,z(t)) V_2(\cos\vartheta(t),\sin\vartheta(t)) 
\end{multline*}
so that we obtain
\begin{equation}\label{between}
0< 2V_1(\cos\vartheta(t),\sin\vartheta(t)) \leq -\vartheta'(t)\leq 2V_2(\cos\vartheta(t),\sin\vartheta(t))\,.
\end{equation}

By the previous computation, recalling~\eqref{period},~\eqref{starstar} and the notation introduced in~\eqref{semilines}, a never-zero solution of~\eqref{system}
\begin{eqnarray}
\text{moving from $l_S^1$ to $l_A^1$ spends a time } \Delta t_1\in[\tau_{1,V_2}\,,\,\tau_{1,V_1}]\,, \label{not1}\\
\text{moving from $l_A^1$ to $l_S^2$ spends a time } \Delta t_2\in[\sigma_{1,V_2}\,,\,\sigma_{1,V_1}]\,, \label{not2}\\
\text{moving from $l_S^2$ to $l_A^2$ spends a time } \Delta t_3\in[\tau_{2,V_2}\,,\,\tau_{2,V_1}]\,, \label{not3}\\
\text{moving from $l_A^2$ to $l_S^1$ spends a time } \Delta t_4\in[\sigma_{2,V_2}\,,\,\sigma_{2,V_1}]\,, \label{not4}\\
\text{completes a rotation around the origin in a time } \Delta t\in [\tau_{V_2}\,,\,\tau_{V_1}]\,. \label{not5}
\end{eqnarray}

\section{Main results}\label{mainresults}
\subsection{Nonresonance}\label{secnonres}

In this section we consider the  boundary value problem
\begin{equation}\label{problem+r}
\begin{cases}
J z' = g(t,z)+p(t,z)\,,\\
z(0)\in l_S\,, \qquad 
z(T)\in l_A\,,
\end{cases}
\end{equation}
where $g:[0,T]\times\R^2\to\R^2$ satisfies Assumption~\ref{ass1} and $p$ satisfies
\begin{assumption}\label{ass+p}
The function $p:[0,T]\times\R^2\to\R^2$ is a Carath\'eodory function such that, for every compact set $K\subset \R^2$, $|p(t,z)|\leq \ell_K(t)$, for a.e. $t\in[0,T]$ and $z\in K$, for a suitable $\ell_K \in L^2(0,T)$. Moreover,
\begin{equation}\label{sublinear}
\lim_{|z|\to \infty} \frac{p(t,z)}{z}=0\,,%  \qquad \text{we say } r(t,z)=o(z)\,,
\end{equation}
uniformly for almost every $t\in[0,T]$.
\end{assumption}

Let us introduce some notations. We will denote by $\Phi:\R\times \R^2 \to \R^2$ the flux of system
\begin{equation}\label{system+r}
J z' = g(t,z)+p(t,z)\,,
\end{equation}
i.e. $\Phi(\cdot ,z_0)$ is the solution $z$ of~\eqref{system+r} such that $z(0)=\Phi(0,z_0)=z_0$.\footnote{We can assume, without loss of generality, the uniqueness of the solutions to the Cauchy problems. Indeed, by standard arguments, all the results in this paper can be obtained with a limit procedure introducing a sequence of approximating nonlinearities having such a uniqueness property.}
We will also consider polar coordinates associated to never-zero solutions $z=\Phi(\cdot,z_0)$: 
\begin{equation}\label{fluxpolar}
\Phi(t,z_0) = \mathscr R(t,z_0) \big( \cos \Theta(t,z_0) \,,\, \sin \Theta(t,z_0) \big)\,.
\end{equation}
For definiteness, we will need to specify the value $\Theta(0,z_0)$ or we will preferably consider the {\sl covered-angle} function 
\begin{equation}\label{coveredangle}
\Delta\Theta(t,z_0)=\Theta(0,z_0)-\Theta(T,z_0)\geq 0\,.
\end{equation}

In this section we are going to prove the following result which generalizes~\cite[Theorem 3.1]{FoGar2013}:
\begin{theorem}[Nonresonance]\label{thm-nonres}
Consider problem~\eqref{problem+r}, where $g$ satisfies Assumption~\ref{ass1} and $p$ satisfies Assumption~\ref{ass+p}. If
$T \notin \mathcal I$, where $\mathcal I$ is the resonance set introduced in~\eqref{union}, then there exists at least one solution of~\eqref{problem+r}.
\end{theorem}

\begin{proof}
If $T\notin\mathcal I$, then we have $T\in(\beta_{j-1},\alpha_j)$ for a certain $j\in\N$ (set $\beta_{-1}=0$). In particular we can find a small $\epsilon>0$ such that 
\begin{equation}\label{inthemiddle}
T\in\big(\beta_{j-1}+(2j-1)\epsilon,\alpha_j-(2j+1)\epsilon\big)\,.
\end{equation}
Introducing polar coordinates $z=\rho e^{i\vartheta}$, the angular velocity of solutions of~\eqref{system+r}
is given by 
\begin{equation*}
-\vartheta'(t)= \frac{\pscal{g(t,z(t))+p(t,z(t))}{z(t)}}{|z(t)|^2} 
\end{equation*}
and we can compute
\begin{equation}\label{between+r}
2V_1(\cos\vartheta(t),\sin\vartheta(t))+ e(t) \leq -\vartheta'(t)\leq 2V_2(\cos\vartheta(t),\sin\vartheta(t)) + e(t) \,.
\end{equation}
where
$$
e(t)=  \frac{\pscal{p(t,z(t))}{z(t)}}{|z(t)|^2} \,. %= o(z(t))\,.
$$

Then, from~\eqref{not1}--\eqref{not5}, for every $\epsilon>0$ we can find $R_1>0$ such that a solution of~\eqref{system+r}, satisfying $|z(t)|\geq R_1$ for every $t\in[0,T]$, rotates clockwise and
\begin{eqnarray}
\text{moving from $l_S^1$ to $l_A^1$ spends a time } \Delta t_1\in(\tau_{1,V_2}-\epsilon\,,\,\tau_{1,V_1}+\epsilon)\,,\qquad \label{not1p}\\
\text{moving from $l_A^1$ to $l_S^2$ spends a time } \Delta t_2\in(\sigma_{1,V_2}-\epsilon\,,\,\sigma_{1,V_1}+\epsilon)\,,\qquad \label{not2p}\\
\text{moving from $l_S^2$ to $l_A^2$ spends a time } \Delta t_3\in(\tau_{2,V_2}-\epsilon\,,\,\tau_{2,V_1}+\epsilon)\,,\qquad \label{not3p}\\
\text{moving from $l_A^2$ to $l_S^1$ spends a time } \Delta t_4\in(\sigma_{2,V_2}-\epsilon\,,\,\sigma_{2,V_1}+\epsilon)\,,\qquad \label{not4p}\\
\text{completes a rotation around the origin in a time } \hspace{30mm} \nonumber \\
\Delta t\in (\tau_{V_2}-4\epsilon\,,\,\tau_{V_1}+4\epsilon)\,.\qquad \label{not5p}
\end{eqnarray}
Then, by the elastic property we can prove the existence of $R_2>R_1$ such that if a solution of~\eqref{system+r} satisfies $|z(\bar t)|\geq R_2$ for a certain $\bar t\in[0,T]$, then $|z(t)|\geq R_1$ for every $t\in[0,T]$.
E.g., we can adopt a guiding curve method as in~\cite{FS2,FSbox,S5}. Lemma 4.2 in~\cite{FS2} works well in our situation.

Summing up, we can conclude with the following result.

\begin{proposition}\label{propnonres}
For every $j\in\N$ and every $\epsilon>0$, there exists $R_{j,\epsilon}>R_2$ such that
\begin{itemize}
\item if $z^1\in l_S^1$ and $|z^1|\geq R_{j,\epsilon}$ then the solution $\Phi(\cdot,z^1)$ covers the angle $j\pi + \Delta\zeta$ in a time $\tau\in(a_j^1-(2j+1)\epsilon,b_j^1+(2j+1)\epsilon)$.
\item if $z^2\in l_S^2$ and $|z^2|\geq R_{j,\epsilon}$ then the solution $\Phi(\cdot,z^2)$ covers the angle $j\pi + \Delta\zeta$ in a time $\tau\in(a_j^2-(2j+1)\epsilon,b_j^2+(2j+1)\epsilon)$.
\end{itemize}
\end{proposition}

As a consequence we can rewrite the previous proposition as follows, cf. Figure~\ref{fig:eta} (set $R_{-1,\epsilon}=0$ for definiteness).

\begin{remark}\label{remnonres}
Fix $R>\max\{R_{j-1,\epsilon},R_{j,\epsilon}\}$. Setting $z^1=R(\cos \zeta_S,\sin \zeta_S)\in l^1_S$ and $z^2=-R(\cos \zeta_S,\sin \zeta_S)\in l^2_S$. Then
$$
\Delta\Theta(T,z^1) \in \big((j-1)\pi+\Delta\zeta\,,\, j\pi+\Delta\zeta\big)\,.
$$
$$
\Delta\Theta(T,z^2) \in \big((j-1)\pi+\Delta\zeta\,,\, j\pi+\Delta\zeta\big)\,.
$$
In particular, $\Phi(T,z^1)$ and $\Phi(T,z^2)$ belongs to different connected components of $\R^2\setminus l_A$.
\end{remark}

By standard argument (cf. Figure~\ref{fig:eta}) the curve $\eta:[-R,R]\to\R^2$ defined as $\eta(\sigma)=\Phi\big(T,\sigma(\cos \zeta_S,\sin \zeta_S) \big)$ intersects the line $l_A$ for a certain $\bar \sigma$ and the proof of Theorem~\ref{thm-nonres} is given.
\end{proof}

\begin{figure}[t]

\centerline{\epsfig{file=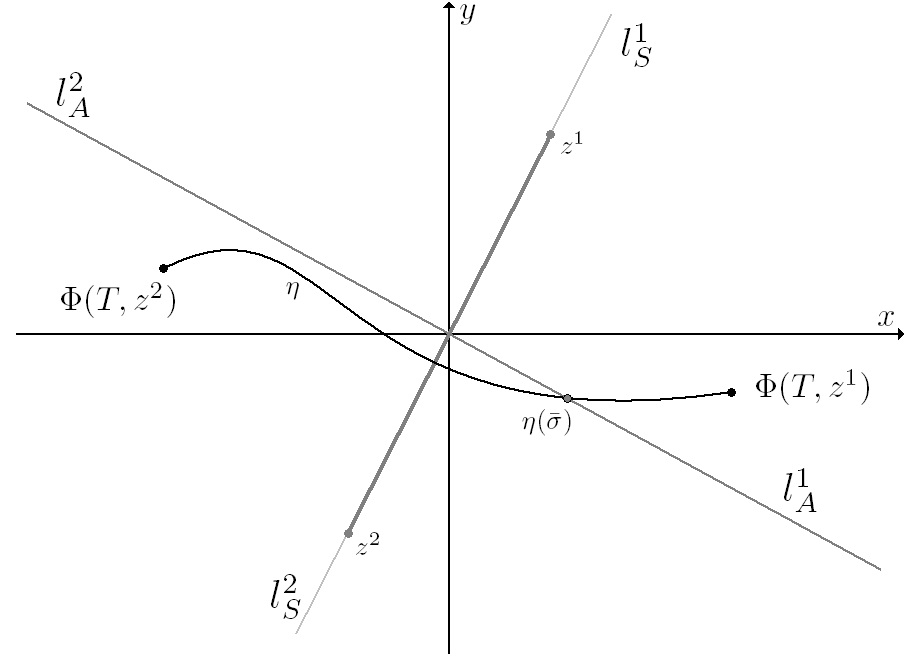, width = 10 cm}}
\caption{A sketch on Remark~\ref{remnonres} and an idea of the proof of Theorem~\ref{thm-nonres}.
%The lines $\ell^S$ and $\ell^A$ and the values introduced in~\eqref{deltazeta},~\eqref{taulaps},~\eqref{semilines}.
}
\label{fig:eta}
\end{figure}

\subsection{Simple resonance}\label{secsimres}

In this section we consider the case $T\in \partial \mathcal I$, where $\mathcal I$ was introduced in~\eqref{union}. In such a situation $\beta_{j-1}< T= \alpha_j$ or $\beta_{j-1}=T<\alpha_j$ for a certain $j\in\N$. 
In order to obtain the existence of a solution to~\eqref{problem+r} we need to introduce a Landesman-Lazer type of condition.
We can distinguish the following cases:
%\begin{eqnarray}
\begin{equation}
T=\alpha_j = a_j^1 < a_j^2\,, \label{case1}\tag{R1} %\\
\end{equation}
\begin{equation}
T=\alpha_j = a_j^2 < a_j^1\,, \label{case2}\tag{R2} %\\
\end{equation}
\begin{equation}
T=\alpha_j = a_j^1 = a_j^2\,; \label{case3}\tag{R3} %\\
\end{equation}
\begin{equation}
T=\beta_{j-1} = b_{j-1}^1 > b_{j-1}^2\,,  \label{case4}\tag{R4} %\\
\end{equation}
\begin{equation}
T=\beta_{j-1} = b_{j-1}^2 > b_{j-1}^1\,,  \label{case5}\tag{R5} %\\
\end{equation}
\begin{equation}
T=\beta_{j-1} = b_{j-1}^1 = b_{j-1}^2\,.  \label{case6}\tag{R6}
\end{equation}
%\end{eqnarray}
%
%\begin{remark}[A comparison with~\cite{BosGar}]
%In~\cite{BosGar} the case $V_1=V_2$ has been studied. In such a situation $a_j^i=b_j^i$ for every $j\in\N$ and $i\in\{1,2\}$. In particular  Theorem 2.1 and Theorem 2.3-2 in~\cite{BosGar} treats the case ``$j$ is  even'', while Theorem 2.3-1 and Theorem 2.3-3 the case ``$j$ is odd''.
%
%E.g., in~\cite[Theorem 2.1]{BosGar} the situations (17),(18),(19) or (20) correspond respectively to~\eqref{case1},~\eqref{case4},~\eqref{case3} and~\eqref{case6}.
%\end{remark}

Let us introduce the following functions, where $G(t,z)=g(t,z)+p(t,z)$:
\begin{multline}\label{LLliminf}
\mathcal J^-(\theta) = \int_0^T \liminf_{(\lambda,\omega)\to(+\infty,\theta)} \big[ \pscal{G(t,\lambda\varphi_{V_1}(t+\omega))}{\varphi_{V_1}(t+\omega)} -  \\
 - 2\lambda  V_1(\varphi_{V_1}(t+\omega))\big]\, dt\,,
\end{multline}
\begin{multline}\label{LLlimsup}
\mathcal J^+(\theta) = \int_0^T \limsup_{(\lambda,\omega)\to(+\infty,\theta)} \big[
 \pscal{G(t,\lambda\varphi_{V_2}(t+\omega))}{\varphi_{V_2}(t+\omega)} -  \\
 - 2\lambda V_2(\varphi_{V_2}(t+\omega)) 
\big]\, dt\,.
\end{multline}
%Notice that $\mathcal J^-$ is $\tau_{V_1}$-periodic, while $\mathcal J^+$ is $\tau_{V_2}$-periodic.

The Landesman-Lazer type of assumption we need to require can be summarized as follows.

\begin{assumption}[Landesman-Lazer type of assumptions]\label{assLL}
%Consider problem~\eqref{problem+r}, where $g$ satisfies Assumption~\ref{ass1} and $r$ satisfies Assumption~\ref{ass+r}. If
%$T\in \mathcal I$ but $T\notin \Inter{\mathcal I}$, then assume
\begin{eqnarray*}
%\text{if $j$ is even and} 
& \text{If~\eqref{case1} holds, assume} & \mathcal J^+(\tau_{0,V_2})<0\,, \\
& \text{if~\eqref{case2} holds, assume} & \mathcal J^+(\tau_{0,V_2}+\tau_{1,V_2}+\sigma_{1,V_2})<0\,, \\
& \text{if~\eqref{case3} holds, assume} & \mathcal J^+(\tau_{0,V_2})<0 \text{ and }  \\
                                    && \mathcal J^+(\tau_{0,V_2}+\tau_{1,V_2}+\sigma_{1,V_2})<0\,;\\
& \text{if~\eqref{case4} holds, assume} & \mathcal J^-(\tau_{0,V_1})>0\,, \\
& \text{if~\eqref{case5} holds, assume} & \mathcal J^-(\tau_{0,V_1}+\tau_{1,V_1}+\sigma_{1,V_1})>0\,, \\
& \text{if~\eqref{case6} holds, assume} & \mathcal J^-(\tau_{0,V_1})>0 \text{ and }  \\
                                    && \mathcal J^-(\tau_{0,V_1}+\tau_{1,V_1}+\sigma_{1,V_1})>0\,. \\
%\text{if $j$ is odd and} & \text{\eqref{case1} holds, assume} & LL \\ IDEM COME SOPRA
%& \text{\eqref{case2} holds, assume} & LL \\
%& \text{\eqref{case3} holds, assume} & LL \\
%& \text{\eqref{case4} holds, assume} & LL \\
%& \text{\eqref{case5} holds, assume} & LL \\
%& \text{\eqref{case6} holds, assume} & LL \\
\end{eqnarray*}
\end{assumption}

In the case of {\sl simple resonance} only one of the alternatives~\eqref{case1}-\eqref{case6} is verified, and so we need to check only one of the alternatives in Assumption~\ref{assLL}.

\begin{remark}[A comparison with~\cite{BosGar}]
In~\cite{BosGar} the case $V_1=V_2$ has been studied. In such a situation $a_j^i=b_j^i$ for every $j\in\N$ and $i\in\{1,2\}$. In particular  Theorem~2.1 and Theorem~2.3-2 in~\cite{BosGar} treats the case ``$j$ is  even'', while Theorem~2.3-1 and Theorem~2.3-3 the case ``$j$ is odd''.

E.g., in~\cite[Theorem 2.1]{BosGar} the situations (17),(18),(19) or (20) correspond respectively to~\eqref{case1}, \eqref{case4}, \eqref{case3} and~\eqref{case6}.
\end{remark}

The existence result is the following.

\begin{theorem}[Simple resonance]\label{thmLL}
Consider problem~\eqref{problem+r}, where $p$ satisfies Assumption~\ref{ass+p} and $g$ satisfies Assumption~\ref{ass1}. If
$T\in \partial \mathcal I$, where $\mathcal I$ was introduced in~\eqref{union}, then there exists at least one solution of~\eqref{problem+r} provided that Landesman-Lazer Assumption~\ref{assLL} is fulfilled.
\end{theorem}

\begin{proof}
Let us start assuming the validity of~\eqref{case1}, 
in particular $\beta_{j-1}<T=\alpha_j$.

Let us consider $z^2\in l_S^2$ and the solution $\Phi(\cdot,z^2)$ of~\eqref{system+r}. By Proposition~\ref{propnonres}, if $|z^2|$ is sufficiently large
then 
\begin{equation}\label{good1}
\Delta\Theta(T,z^2)\in \big((j-1)\pi + \Delta\zeta\,,\, j\pi + \Delta\zeta \big)\,.
\end{equation}
In fact, $b_{j-1}^2<T<a_j^2$ holds.

Let us now consider $z^1\in l_S^1$ and the solution $\Phi(\cdot,z^1)$ of~\eqref{system+r}. Arguing similarly, if $|z^1|$ is sufficiently large then 
\begin{equation}\label{bad1}
\Delta\Theta(T,z^1)\in \big((j-1)\pi + \Delta\zeta\,,\, (j+1)\pi + \Delta\zeta \big)\,.
\end{equation}
Notice that the interval is larger since $T=a_j^1$. We need to prove that the situation $\Delta\Theta(T,z^1)\in \big[j\pi + \Delta\zeta\,,\, (j+1)\pi + \Delta\zeta \big)$ is forbidden. Once this claim is proved, we will obtain the existence of $R>0$ with the property explained in Remark~\ref{remnonres}, thus permitting us to conclude the proof of the theorem as in the previous section.

We argue by contradiction and suppose the existence of a sequence $(z_n^0)_n\subset l_S^1$, with $|z_n^0|\to \infty$, such that
\begin{equation}\label{absurd}
\Delta\Theta(T,z_n^0)\in \big[j\pi + \Delta\zeta\,, (j+1)\pi + \Delta\zeta \big)\,.
\end{equation}
%{\tiny
%\begin{eqnarray}
%\label{PiA+}
%\Pi^+_A = \{ (x,y)\in\R^2 \mid \cos \zeta_A y - \sin \zeta_A x \geq 0 \}\,,\\
%\label{PiA-}
%\Pi^-_A = \{ (x,y)\in\R^2 \mid \cos \zeta_A y - \sin \zeta_A x \leq 0 \}\,.
%\end{eqnarray}
%{\sl A standard application of Gronwall's Lemma gives $\min_{[0,T]} |z_n(t)| = \infty$.}
%}
Set $z_n(t)=\Phi(t,z_n^0)$ and introduce the sequence
\begin{equation}\label{whoiswn}
w_n= \frac{z_n}{\|z_n\|_\infty}=\frac{\Phi(\cdot,z_n^0)}{\|\Phi(\cdot,z_n^0)\|_\infty}
\end{equation}
consisting of solutions to
\begin{equation}\label{seq}
\begin{cases}
\ds J w_n ' = (1-\Gamma_n(t)) \nabla V_1(w_n) + \Gamma_n(t) \nabla V_2(w_n)
+\frac{p(t,w_n(t)\|z_n\|_\infty)}{\|z_n\|_\infty}\,,\\
w_n(0)\in l_S^1\,,
\end{cases}
\end{equation}
where $\Gamma_n(t)=\gamma(t,w_n(t)\|z_n\|_\infty)$. Hence, $(w_n)_n$ is bounded in $H^1(0,T)$, so that up to a subsequence it converges uniformly and weakly in $H^1(0,T)$ to a certain non-trivial function $w$ satisfying $w(0)\in l_S^1$. Moreover, the sequence $(\Gamma_n)_n$ is bounded in $L^2$ and converges weakly to a certain $\Gamma\in L^2(0,T)$, up to a subsequence. The sequence is contained 
in the closed convex subset $\{ q\in L^2(0,T) \mid 0\leq q(t)\leq 1 \text{ a.e. in } [0,T] \}$ so that we have also $0\leq \Gamma \leq 1$ a.e. in $[0,T]$. Then, passing to the weak limit in~\eqref{seq} we get
\begin{equation}\label{seqlimit}
\begin{cases}
J w' = (1-\Gamma(t)) \nabla V_1(w) + \Gamma(t) \nabla V_2(w)\,,\\
w(0)\in l_S^1\,.
\end{cases}
\end{equation}

\medbreak

We claim that
% $\Gamma\equiv0$ or 
$\Gamma\equiv 1$ a.e. in $[0,T]$.

\medbreak

From~\eqref{absurd}, setting $w_n^0=w_n(0)=\frac{z_n^0}{\|z_n\|_\infty}$ and using polar coordinates $w_n=\varrho_n e^{i\theta_n}$,

we get $\theta_n(0)-\theta_n(T)
%\Delta\Theta(T,w_n^0)
\in \big[j\pi + \Delta\zeta\,,\, (j+1)\pi + \Delta\zeta \big)$. Hence, setting % $w^0=w(0)$ and 
$w=\varrho e^{i\theta}$,
\begin{equation}\label{anglelim}
%\Delta\Theta(T,w^0)
\theta(0)-\theta(T)
\in \big[j\pi + \Delta\zeta\,,\, (j+1)\pi + \Delta\zeta \big]\,,
\end{equation}
so that there exists $\bar t \in [0,T]$ such that $\theta(0)-\theta(\bar t) = j\pi + \Delta\zeta$.
% Passing to polar coordinate $w=\varrho e^{i\theta}$, as in~\eqref{between},
Hence, integrating
$$
- \frac{\theta'(t)}{2V_2(\cos\theta(t),\sin\theta(t))} \leq 1 %\leq - \frac{\theta'(t)}{2V_1(\cos\theta(t),\sin\theta(t))} 
$$
in the interval $[0,\bar t]$ we get
\begin{equation}\label{int01}
\int^{\zeta_S}_{\zeta_S-(j\pi + \Delta\zeta)} \frac{d\theta}{2V_2(\cos\theta,\sin\theta)}=a_j^1 \leq \bar t 
\end{equation}
By hypothesis $a_j^1=T$, hence $\bar t = T$ holds.

We have proved that $w(T)\in l_A$ and $w$ cover the angle $j\pi+\Delta\zeta$ in the interval $[0,T]$.

Let us now parametrize $w$ using the polar coordinates induced by $\varphi_{V_2}$: 
\begin{equation}\label{polmodV2}
w(t)=r(t)\varphi_{V_2}(t+\omega(t))	\,.
\end{equation}
A standard computation provide
\begin{equation}\label{radV2}
r'(t) = -r(t) (1-\Gamma(t)) \pscal{\nabla V_1(\varphi_{V_2}(t+\omega(t)))}{\varphi_{V_2}'(t+\omega(t))}\,,
\end{equation}
\begin{equation}\label{angV2}
\omega'(t)=(1-\Gamma(t))(2V_1(\varphi_{V_2}(t+\omega(t)))-1)\,.
\end{equation}
Notice that $\omega(0)=\tau_{0,V_2}$ and $\omega(T)=\tau_{0,V_2}+a_j^1-T=\tau_{0,V_2}$, so that
\begin{equation}\label{intzero}
0 = \int_0^T \omega'(t) \,dt = \int_0^T (1-\Gamma(t))(2V_1(\varphi_{V_2}(t+\omega(t)))-1)\,dt\,.
\end{equation}
Recalling that $V_1\leq V_2$, we get $2V_1(\varphi_{V_2})-1 \leq 2V_2(\varphi_{V_2})-1 = 0$ so that
$$
(1-\Gamma(t))(2V_1(\varphi_{V_2}(t+\omega(t)))-1) \leq 0 \qquad \text{a.e. in } [0,T]\,.
$$
Hence, from~\eqref{intzero}, we necessarily have
\begin{equation}\label{itiszero}
(1-\Gamma(t))(2V_1(\varphi_{V_2}(t+\omega(t)))-1) = 0 \qquad \text{a.e. in } [0,T]\,,
\end{equation}
in particular $\omega\equiv \tau_{0,V_2}$ a.e. in $[0,T]$.

Let us now focus our attention on the radial velocity formula~\eqref{radV2}. 
We are going to prove that $r'\equiv0$ almost everywhere in $[0,T]$.
Let us consider $t_0\in[0,T]$ such that $\Gamma(t_0)<1$ (the situation is trivial if $\Gamma(t_0)=1$).
By~\eqref{itiszero} we necessarily have $2V_1(\varphi_{V_2}(t_0+\omega(t_0)))=1$. Recalling that $V_2\geq V_1$ and $2V_2(\varphi_{V_2})\equiv 1$, we find that $t_0$ is a minimum of the function $\mathcal V(t)= V_2(\varphi_{V_2}(t+\omega(t)))-V_1(\varphi_{V_2}(t+\omega(t)))$, precisely $\mathcal V(t_0)=0$, so that $\mathcal V'(t_0)=0$. Being $V_2$ constant along $\varphi_{V_2}$ we get $\left.\frac{d}{dt}V_1(\varphi_{V_2}(\cdot+\omega))\right|_{t=t_0}=0$ and consequently
$\pscal{\nabla V_1(\varphi_{V_2}(t_0+\omega(t_0)))}{\varphi_{V_2}'(t_0+\omega(t_0))} =0$, giving $r'(t_0)=0$.

\medbreak

We have proved that
$w= C \varphi_{V_2}(t+\tau_{0,V_2})$ for a certain constant $C>0$,
so that $\Gamma\equiv 1$ a.e. in $[0,T]$.

\medbreak

Let us consider again the sequence $(w_n)_n$ introduced in~\eqref{whoiswn} and the polar coordinates $w_n = \varrho_n e^{i\theta_n}$. We have
%We are going to prove that 
\begin{equation}\label{geq}
\int_{\theta_n(T)}^{\theta_n(0)} \frac{d\theta}{2V_2(\cos\theta,\sin\theta)} \geq T  \,,
\end{equation}
for large indexes $n$
%Once such a inequality is proved, then we get a contradiction. By construction, from 
by the validity of~\eqref{absurd}:
indeed, the integral provides the time spent by a solution of the system $Jz'=\nabla V_2(z)$ to cover the angular sector between $\theta_n(0)=\zeta_S$ and $\theta_n(T)\leq \zeta_S-(\pi j +\Delta\zeta)$, while $T$ is the time spent to cover the (not larger) angular sector between $\zeta_S$ and $\zeta_S-(\pi j +\Delta\zeta)$.

The angular speed of $w_n$ is given by
\begin{eqnarray*}
-\theta_n'(t) &=& \frac{\pscal{Jw_n'(t)}{w_n(t)}}{|w_n(t)|^2} = \frac{\pscal{Jz_n'(t)}{z_n(t)}}{|z_n(t)|^2}\\
&=& 2V_2(\cos\theta_n(t),\sin\theta_n(t)) \\
&& \hspace{20mm}+ \frac{\pscal{g(t,z_n(t))+p(t,z_n(t)) - \nabla V_2(z_n(t))}{z_n(t)}}{|z_n(t)|^2}\,,
\end{eqnarray*}
thus giving
$$
\int_{\theta_n(T)}^{\theta_n(0)} \frac{d\theta}{2V_2(\cos\theta,\sin\theta)} = T + \int_0^T \frac{\pscal{\mathcal G(t,z_n(t))}{z_n(t)}}{2V_2(z_n(t))}\,dt\,,
$$
where $\mathcal G(t,z)=g(t,z)+p(t,z)-\nabla V_2(z)$. By~\eqref{geq}, we obtain
$$
\mathcal X_n := \int_0^T \frac{\pscal{\mathcal G(t,z_n(t))}{z_n(t)}}{2V_2(z_n(t))} \geq 0
$$
for sufficiently large indexes $n$. We parametrize the solutions $z_n$ in the polar coordinates induced by $\varphi_{V_2}$, i.e. we set
$$
z_n(t)= r_n(t) \varphi_{V_2}(t+\omega_n(t))\,.
$$
So, we obtain %Euler's formula
$$
\mathcal X_n = \int_0^T \frac{\pscal{\mathcal G\big(t,r_n(t)\varphi_{V_2}(t+\omega_n(t))\big)}{\varphi_{V_2}(t+\omega_n(t))}}{r_n(t)} 
$$
From ~\eqref{whoiswn} and recalling that $w_n \to w=C\varphi_{V_2}(t+\tau_{0,V_2})$ uniformly, we have $\frac{r_n(t)}{\|z_n\|_\infty} \to C$ and $\omega_n\to \tau_{0,V_2}$. Then
\begin{eqnarray*}
0&\leq& \limsup_{n\to+\infty} \|z_n\|_{\infty} \mathcal X_n\\
&\leq& \int_0^T \limsup_{n\to+\infty} \frac{\pscal{\mathcal G(t,r_n(t)\varphi_{V_2}(t+\omega_n(t)))}{\varphi_{V_2}(t+\omega_n(t))}}{\frac{r_n(t)}{\|z_n\|_\infty}} \, dt\\
&\leq& \frac1C \int_0^T \limsup_{(\lambda,\omega)\to(+\infty,\tau_{0,V_2})} \pscal{\mathcal G(t,\lambda\varphi_{V_2}(t+\omega))}{\varphi_{V_2}(t+\omega)} \, dt
\end{eqnarray*}
The last inequality, using~\eqref{LLlimsup}, can be rewritten as $\mathcal J^+(\tau_{0,V_2}) \geq 0$ which contradicts Assumption~\ref{assLL}.

\medbreak

We have proved the theorem if~\eqref{case1} holds.

\medbreak

Let us now spend few words in order to explain how to adapt the proof in the other situations.

\medbreak

Assume~\eqref{case2}. In such a situation the role of $l_S^1$ and $l_S^2$ is switched.
So, the ``good estimate'' in~\eqref{good1} is easily obtained for $z^1\in l_S^1$, with $|z^1|$ sufficiently large, while the ``bad estimate''~\eqref{bad1} occurs treating $z^2\in l_S^2$. In this case, we assume by contradiction
the existence of a sequence $(z_n^0)_n\subset l_S^2$, with $|z_n^0|\to \infty$, satisfying~\eqref{absurd}.
The proof can be plainly adapted, but we underline the main differences: the starting angle $\Theta(0,z_n^0)=\zeta_S$ is replaced by $\Theta(0,z_n^0)=\zeta_S+\pi$ and the constant $a_j^1$ is replaced by $a_j^2$. In particular,~\eqref{int01} becomes
$$%\begin{equation}\label{int01b}
\int^{\zeta_S+\pi}_{\zeta_S+\pi-(j\pi + \Delta\zeta)} \frac{d\theta}{2V_2(\cos\theta,\sin\theta)}=a_j^2 \leq \bar t \,.
$$%\end{equation}
Finally, the limit function is now $w(t)=C\varphi_{V_2}(t+\tau_{0,V_2}+\tau_{1,V_2}+\sigma_{1,V_2})$.

\medbreak

Assume now~\eqref{case3}. Under this hypothesis, we get the ``bad estimates''~\eqref{bad1} both for $z_0\in l_S^1$ and $z_0\in l_S^2$ with $|z_0|$ large so that the proof is a {\sl glueing} of cases~\eqref{case1} and~\eqref{case2}. The same reasoning holds for~\eqref{case6}: the proof will follow by the ones of cases~\eqref{case4} and~\eqref{case5} we are going to provide.

\medbreak

Let us consider~\eqref{case4}. The validity of~\eqref{good1} is given when we treat solutions $\Phi(\cdot,z_0)$ of~\eqref{system+r} with  $z_0\in l_S^2$ with $|z_0|$ sufficiently large. However, solutions $\Phi(\cdot,z_0)$ of~\eqref{system+r} with  $z_0\in l_S^1$ satisfies (no more~\eqref{bad1}, but)
\begin{equation}\label{bad1b}
\Delta\Theta(T,z_0)\in \big((j-2)\pi + \Delta\zeta\,,\, j\pi + \Delta\zeta \big)\,.
\end{equation}
and we need to forbid the situation $\Delta\Theta(T,z_0)\in \big((j-2)\pi + \Delta\zeta\,,\, (j-1)\pi + \Delta\zeta \big]$. Arguing as above, we can consider a diverging sequence $(z_n^0)_n\subset \ell_S^1$ such that
$$%\begin{equation}\label{bad1}
\Delta\Theta(T,z_n^0)\in \big((j-2)\pi + \Delta\zeta\,,\, (j-1)\pi + \Delta\zeta \big]\,.
$$%\end{equation}
We introduce similarly the sequence $(w_n)_n$, and prove that it converges to a solution $w$ of~\eqref{seqlimit}. Now, we claim that $\Gamma\equiv 0$. In this case~\eqref{anglelim} is replaced by
\begin{equation}\label{anglelim2}
\Delta\Theta(T,w^0)\in \big[(j-2)\pi + \Delta\zeta\,,\, (j-1)\pi + \Delta\zeta \big]\,.
\end{equation}
Then, we introduce $\bar t >T$ such that $\Delta\Theta(\bar t ,w^0)=(j-1)\pi + \Delta\zeta$ and compute, since $- \frac{\theta'(t)}{2V_1(\cos\theta(t),\sin\theta(t))}\geq 1 $, the validity of 
$$%\begin{equation}\label{int01}
\int^{\zeta_S}_{\zeta_S-((j-1)\pi + \Delta\zeta)} \frac{d\theta}{2V_1(\cos\theta,\sin\theta)}=b_j^1 \geq \bar t 
$$%\end{equation}
bringing us to the conclusion $T=b_j^1 \geq \bar t \geq T$, hence $T=\bar t$.

Once proved that $w$ cover the angle $(j-1)\pi + \Delta\zeta$ in the interval $[0,T]$, we can introduced 
the polar coordinates induced by $\varphi_{V_1}$: 
\begin{equation}\label{polmodV1}
w(t)=r(t)\varphi_{V_1}(t+\omega(t))	\,.
\end{equation}
A standard computation provides
\begin{equation}\label{radV1}
r'(t) = -r(t) \Gamma(t) \pscal{\nabla V_2(\varphi_{V_1}(t+\omega(t)))}{\varphi_{V_1}'(t+\omega(t))}\,,
\end{equation}
\begin{equation}\label{angV1}
\omega'(t)=\Gamma(t)(2V_2(\varphi_{V_1}(t+\omega(t)))-1)\,.
\end{equation}
Again $\omega(0)=\omega(T)=\tau_{0,V_1}$ and we can similarly conclude that $\Gamma\equiv0$ almost everywhere in $[0,T]$, thus obtaining
$w(t)= C \varphi_{V_1}(t+\tau_{0,V_1})$ for a suitable positive constant $C$. Then, 
introducing standard polar coordinates for the sequence~$(w_n)_n$, we can prove that
$$
\int_{\theta_n(T)}^{\theta_n(0)} \frac{d\theta}{2V_1(\cos\theta,\sin\theta)} \leq T
$$
for large indexes $n$. Similarly as above we can compute
$$
\int_{\theta_n(T)}^{\theta_n(0)} \frac{d\theta}{2V_1(\cos\theta,\sin\theta)} = T + \int_0^T \frac{\pscal{\mathcal G(t,z_n(t))}{z_n(t)}}{2V_1(z_n(t))} =: T + \mathcal X_n\,,
$$
where now $\mathcal G(t,z)=g(t,z)+p(t,z)-\nabla V_1(z)$ and $\mathcal X_n\leq 0$ for every $n$.
Parametrizing the solutions $z_n$ in the polar coordinates induced by $\varphi_{V_1}$, i.e. 
$z_n(t)= r_n(t) \varphi_{V_1}(t+\omega_n(t))$
we obtain 
\begin{eqnarray*}
0&\geq& \liminf_{n\to+\infty} \|z_n\|_{\infty} \mathcal X_n\\
&\geq& \int_0^T \liminf_{n\to+\infty} \frac{\pscal{\mathcal G(t,r_n(t)\varphi_{V_1}(t+\omega_n(t)))}{\varphi_{V_1}(t+\omega_n(t))}}{\frac{r_n(t)}{\|z_n\|_\infty}} \, dt\\
&\geq& \frac1C \int_0^T \liminf_{(\lambda,\omega)\to(+\infty,\tau_{0,V_1})} \pscal{\mathcal G(t,\lambda\varphi_{V_1}(t+\omega))}{\varphi_{V_1}(t+\omega)} \, dt\,.
\end{eqnarray*}
Finally, the last inequality, using~\eqref{LLliminf}, can be rewritten as $\mathcal J^-(\tau_{0,V_2}) \leq 0$ which contradicts Assumption~\ref{assLL}.

\medbreak

This prove the case~\eqref{case4}.

\medbreak

The modification needed to prove the case~\eqref{case5} from the previous situation, are similar to the ones provided when~\eqref{case2} holds.
\end{proof}

\subsection{Double resonance}\label{secdoubleres}

In this section we consider the case
$T\in \partial\widetilde{\mathcal I}\setminus \partial \mathcal I$,
where $\mathcal I$ and $ \widetilde{\mathcal I}$ were introduced in~\eqref{union} and~\eqref{uniontilde}. In particular, $T= \beta_{j-1}=\alpha_{j}$ for a certain $j\in\N$.

In such a situation one among the alternatives~\eqref{case1}--\eqref{case3} is fulfilled and one among~\eqref{case4}--\eqref{case6}. We thus need to introduce a double Landesman-Lazer type of condition in order to find solutions to~\eqref{problem+r}, i.e. the validity of two of the requirements in Assumption~\ref{assLL} is necessary.

\begin{theorem}[Double resonance]\label{thmLL2}
Consider the problem~\eqref{system+r}, where $g$ satisfies Assumption~\ref{ass1} and $p$ satisfies Assumption~\ref{ass+p}. If
$T\in \partial\widetilde{\mathcal I}\setminus \partial \mathcal I$, then there exists at least one solution of~\eqref{problem+r} provided that Landesman-Lazer Assumptions~\ref{assLL} are fulfilled.
\end{theorem}

\begin{proof}
We focus our attention on the situation which presents all the difficulties: we thus assume that both~\eqref{case3} and~\eqref{case6} hold.

Let us consider $z_0\in l_S=  l_S^1\cup l_S^2$ and the solution $\Phi(\cdot,z_0)$ of~\eqref{system+r}. If $|z_0|$ is sufficiently large, we have (no more~\eqref{bad1} or~\eqref{bad1b}, but)
\begin{equation}\label{bad2}
\Delta\Theta(T,z_0)\in \big((j-2)\pi + \Delta\zeta\,,\, (j+1)\pi + \Delta\zeta \big)\,.
\end{equation}

We need to avoid the situation
$$
\Delta\Theta(T,z_0)\in
 \big((j-2)\pi + \Delta\zeta\,,\, (j-1)\pi + \Delta\zeta \big]\cup
\big[j\pi + \Delta\zeta\,,\, (j+1)\pi + \Delta\zeta \big)\,.
$$
for every $z_0\in l_S$ such that $|z_0|$ is sufficiently large.

We assume the existence of four diverging sequences: $(z_{n,1}^0)_n\subset\ell_S^1$, $(z_{n,2}^0)_n\subset\ell_S^1$, $(z_{n,3}^0)_n\subset\ell_S^2$, $(z_{n,4}^0)_n\subset\ell_S^2$, such that
\begin{eqnarray*}
&&\Delta\Theta(T,z_{n,1}^0)\in \big((j-2)\pi + \Delta\zeta\,,\, (j-1)\pi + \Delta\zeta \big] \,,\\
&&\Delta\Theta(T,z_{n,2}^0)\in \big[j\pi + \Delta\zeta\,,\, (j+1)\pi + \Delta\zeta \big) \,,\\
&&\Delta\Theta(T,z_{n,3}^0)\in \big((j-2)\pi + \Delta\zeta\,,\, (j-1)\pi + \Delta\zeta \big] \,,\\
&&\Delta\Theta(T,z_{n,4}^0)\in \big[j\pi + \Delta\zeta\,,\, (j+1)\pi + \Delta\zeta \big) \,.
\end{eqnarray*}
For all the four sequences, thanks to the validity of the Landesman-Lazer Assumption~\ref{assLL}, adapting the proof of Theorem~\ref{thmLL} we will get a contradiction.

In this way, we obtain the existence of $R>0$ such that, if $z_0\in l_S^2\cup l_S^1$ and $|z_0|>R$ then the solution $\Phi(\cdot,z_0)$ of~\eqref{system+r} satisfies
\begin{equation*}
\Delta\Theta(T,z_0)\in \big((j-1)\pi + \Delta\zeta\,,\, j\pi + \Delta\zeta \big)\,.
\end{equation*}
Hence, we are in the situation of Remark~\ref{remnonres} and we can conclude as in the proof of Theorem~\ref{thm-nonres}.
\end{proof}

\section{Application to asymmetric nonlinearities}\label{secapp}

In this section we focus our attention on scalar differential equations
\begin{equation}\label{sc2}
x''+f(t,x)=0\,,
\end{equation}
where $f$ is a continuous function satisfying
\begin{equation}\label{fm}
0<\nu_1 \leq\liminf_{x\to-\infty} \frac{f(t,x)}{x}\leq\limsup_{x\to-\infty} \frac{f(t,x)}{x} \leq \nu_2\,,
\end{equation}
\begin{equation}\label{fp}
0<\mu_1 \leq\liminf_{x\to+\infty} \frac{f(t,x)}{x}\leq\limsup_{x\to+\infty} \frac{f(t,x)}{x} \leq \mu_2\,,
\end{equation}
uniformly with respect to $t$ (we assume $f$ to be continuous just to simplify the argument).

We address the reader to~\cite[Section 3]{BosGar} for a comparison with the case $\nu=\nu_1=\nu_2$, $\mu=\mu_1=\mu_2$. In this section we extend the results presented there. In particular, we will focus our attention only on the Dirichlet problems. The case of problems with Neumann boundary conditions or mixed boundary conditions $x'(0)=x(T)=0$, which are treated in~\cite{BosGar} too, is left to the reader as an exercise, for briefness.

\medbreak

Setting $z(t)=(x(t),x'(t))$, we can write equation~\eqref{sc2} in the form of a planar system as in~\eqref{plan-mine} where $2V_i=x'(t)^2+\mu_i (x^+(t))^2 +\nu_i (x^-(t))^2 $. The planar system $Jz'=\nabla V_i(z)$, is nothing else but the asymmetric oscillator
$$
x''+ \mu_i x^+ - \nu_i x^- =0\,,
$$
and admits periodic solutions of period $\tau_{V_i} = \frac{\pi}{\sqrt{\mu_i}}+\frac{\pi}{\sqrt{\nu_i}}$ of the form $z(t)=(x(t),x'(t))$ with $x(t)=C\phi_{\mu_i,\nu_i}(t+t_0)$ where $C\in \R^+,\, t_0\in\R$ and
$$
\phi_{\mu_i,\nu_i}(t) :=
\begin{cases}
\ds \frac{1}{\sqrt{\mu_i}} \sin \left(\sqrt{\mu_i} t\right) & \ds t\in \left[0,\frac{\pi}{\sqrt{\mu_i}}\right]\,,\\ \\
\ds \frac{1}{\sqrt{\nu_i}} \sin \left(\sqrt{\nu_i} \left(\frac{\pi}{\sqrt{\mu_i}} -t\right)\right) & \ds t\in \left[\frac{\pi}{\sqrt{\mu_i}}, \frac{\pi}{\sqrt{\mu_i}}+\frac{\pi}{\sqrt{\nu_i}}\right]\,.
\end{cases}
$$
Concerning the problem with Dirichlet boundary conditions, we can compute the constants introduced in Section~\ref{auto}:
$$
\tau_{0,V_i}=\sigma_{1,V_i}=\sigma_{2,V_i}=0\,, \quad \tau_{1,V_i}= \frac{\pi}{\sqrt\mu_i}\,, \quad \tau_{2,V_i}= \frac{\pi}{\sqrt\nu_i}\,.
$$
In particular, %the Dirichlet problem
%associated to $Jz'=\nabla V_i(z)$,
%which is equivalent to
the scalar problem
\begin{equation}\label{ASD}
\begin{cases}
x''+ \mu_i x^+ - \nu_i x^- =0\,,\\
x(0)=0=x(T)\,,
\end{cases}
\end{equation}
has nontrivial solutions if one of the following identities holds for a certain $k\in\N$:
\begin{eqnarray}
&&T= \alpha_{i,k}:=\pi \left[(k+1) \frac{1}{\sqrt{\mu_i}} + k \frac{1}{\sqrt{\nu_i}}\right]\,, \label{DR1}\\
&&T= \beta_{i,k}:=\pi \left[k \frac{1}{\sqrt{\mu_i}} + (k+1) \frac{1}{\sqrt{\nu_i}}\right]\,, \label{DR2}\\
&&T= \gamma_{i,k}:=\pi (k+1) \left[ \frac{1}{\sqrt{\mu_i}} +  \frac{1}{\sqrt{\nu_i}}\right]\,. \label{DR3}
\end{eqnarray}
They are indeed equivalent to~\eqref{resonancesetV1}-\eqref{resonancesetV4}.
In particular, $\phi_{\mu_i,\nu_i}$ solves~\eqref{ASD} when~\eqref{DR1} or~\eqref{DR3} holds, while 
$$
\psi_{\mu_i,\nu_i}(t) := \phi_{\mu_i,\nu_i}\left(t+\pi/\sqrt{\mu_i}\right)
$$
solves it when~\eqref{DR2} or~\eqref{DR3} holds.

Let us now focus on the Dirichlet problem
\begin{equation}\label{psc2}
\begin{cases}
x''+f(t,x)=0\,,\\
x(0)=0=x(T)\,,
\end{cases}
\end{equation}
where $f$ satisfies~\eqref{fm} and~\eqref{fp}.
The resonance set $\mathcal I$ in~\eqref{union} is now
$$
\mathcal I = \bigcup_{k>0}
%\left[k\left(\frac{\pi}{\sqrt{\mu_2}}+\frac{\pi}{\sqrt{\nu_2}} \right)+ \frac{\pi}{\sqrt{\max\{\mu_2,\nu_2\}}}\,,\,
%k\left(\frac{\pi}{\sqrt{\mu_1}}+\frac{\pi}{\sqrt{\nu_1}} \right)+ \frac{\pi}{\sqrt{\min\{\mu_1,\nu_1\}}}\right] \cup 
%\left[(k+1)\left(\frac{\pi}{\sqrt{\mu_2}}+\frac{\pi}{\sqrt{\nu_2}} \right)\,,\,
%(k+1)\left(\frac{\pi}{\sqrt{\mu_1}}+\frac{\pi}{\sqrt{\nu_1}} \right)\right]
[\min\{\alpha_{2,k},\beta_{2,k}\}, \max \{\alpha_{1,k},\beta_{1,k}\} ] \cup [\gamma_{2,k},\gamma_{1,k}]\,.
$$
The resonance set $\mathcal I$ is useful when we want to investigate resonance phenomena when $T$ varies and the constants $\mu_1,\mu_2,\nu_1,\nu_2$ are fixed a priori. On the contrary, fixing $T$, we can study resonance when the other constants change using the set $\Sigma_D$ known as the Dancer-Fu\v c\'ik spectrum associated to~\eqref{ASD}.
 We recall that, for a fixed $T$, the set $\Sigma_D$ collects all the couples $(\mu,\nu)$ in the first quadrant $Q=(\R^+)^2$ satisfying one among~\eqref{DR1}-\eqref{DR3},
and it consists of an infinite number of curves, see Figure~\ref{DFD},
$$
C_{a,b}= \left\{(\mu,\nu)\in Q \mid \frac{a}{\sqrt{\mu}} +  \frac{b}{\sqrt{\nu}} = \frac{T}{\pi}\right\}\,,
$$
where 
\begin{eqnarray*}
&&(a,b)\in \Gamma := \{(a,b)\in \N^2 \mid a+b>0\,, |a-b|\leq 1 \} \\
&&\phantom{(a,b)\in \Gamma :=}= \{(0,1),(1,0),(1,1),(1,2),(2,1),(2,2),\ldots\}\,.
\end{eqnarray*}
Notice that the curves $C_{k,k+1}$ and $C_{k+1,k}$ intersect in the point $(\lambda_k,\lambda_k)$, with $\lambda_k=(\pi (2k+1)/T)^2$.
Hence, we can write
$$
\Sigma_D = \bigcup_{(a,b)\in \Gamma} C_{a,b}\,.
$$

\begin{figure}[t]
\centerline{\epsfig{file=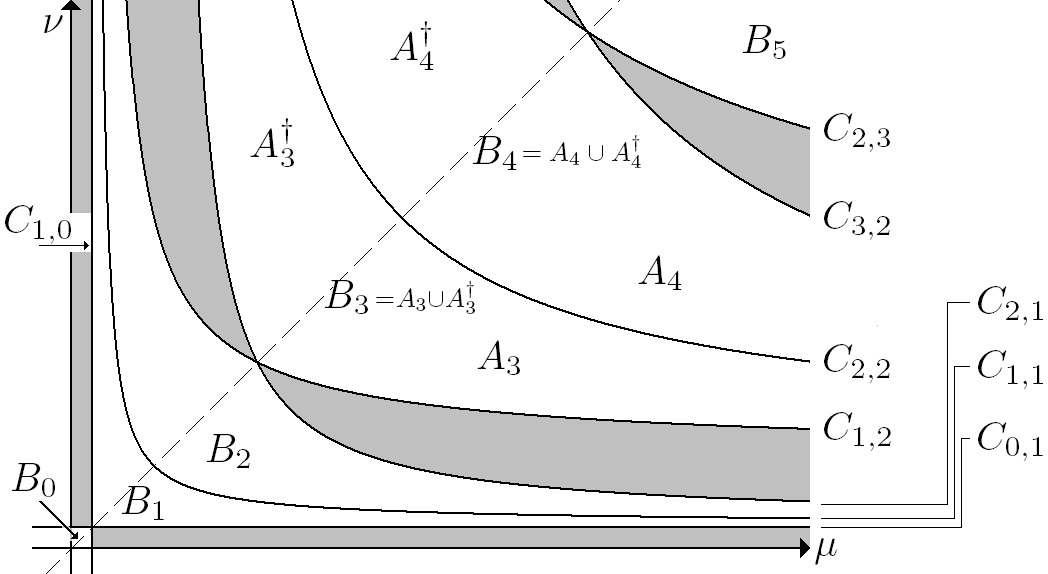, width = 11.5 cm}}
\caption{The Dancer-Fu\v c\'ik spectrum for the Dirichlet problem. The open rectangle $\mathcal R = (\mu_1,\mu_2)\times (\nu_1,\nu_2)$ can't intersect the forbidden region in grey and the spectrum lines $C_{a,b}$. Besides, the vertices of $\mathcal R$ can intersect the spectrum.}
\label{DFD}
\end{figure}

Let us introduce the sets
$$
A_0= \left\{ (\mu,\nu)\in Q \mid \mu\geq\nu\,, \pi >T\sqrt\mu\,\right\}
$$
and, for every integer $j\geq 1$,
$$
A_{2j-1} = \left\{ (\mu,\nu)\in Q \,\Big|\, \mu\geq\nu\,;\, j \left(\frac{\pi}{\sqrt{\mu}}+\frac{\pi}{\sqrt{\nu}}\right)-\frac{\pi}{\sqrt{\mu}} <T<
j \left(\frac{\pi}{\sqrt{\mu}}+\frac{\pi}{\sqrt{\nu}}\right) \right\}\,,
$$
$$
A_{2j} = \left\{ (\mu,\nu)\in Q \,\Big|\, \mu\geq\nu\,;\, j \left(\frac{\pi}{\sqrt{\mu}}+\frac{\pi}{\sqrt{\nu}}\right) <T<
j \left(\frac{\pi}{\sqrt{\mu}}+\frac{\pi}{\sqrt{\nu}}\right)+\frac{\pi}{\sqrt{\mu}} \right\}\,.
$$
Then we define $A_k^\dagger=\{(\mu,\nu)\in Q \mid (\nu,\mu)\in A_k\}$, which is the specular of $A_k$ with respect to the line $\mu=\nu$.
Finally, define $B_k=A_k\cup A_k^\dagger$ and $B=\bigcup_{k\geq0} B_k$, see Figure~\ref{DFD}, and the open rectangle
%$$
%B= \bigcup_{k\in\N} B_k =  \bigcup_{k\in\N} A_k \cup A_k^\dagger
%$$
$$
\mathcal R = (\mu_1,\mu_2)\times(\nu_1,\nu_2)\,.
$$

Concerning the resonance phenomenon, 
we can summarize the possible situations verifying the position of the open rectangle $\mathcal R$ with respect to the Dancer-Fu\v c\'ik spectrum as follows:
\begin{itemize}
\item {\sl Nonresonance}: $\overline{\mathcal R} \subset B$,
\item {\sl Resonance}: ${\mathcal R} \subset B$ and $\overline{\mathcal R} \cap \Sigma_D\neq \varnothing$. In particular, in such a situation, one ({\sl simple resonance}) or both ({\sl double resonance}) the points $(\mu_1,\nu_1)$ and $(\mu_2,\nu_2)$ belongs to $\Sigma_D$.
\end{itemize}
In particular, we need to avoid the situation ``$\mathcal R$ {\sl intersects} $Q \setminus(B\cup \Sigma_D)$'' (roughly speaking, $\mathcal R$ can't intersect the grey-coloured region in Figure~\ref{DFD}).

\medbreak

We now focus our attention on the Landesman-Lazer conditions we need to introduce. At first we define the values
\begin{multline*}
\mathcal A^-(\zeta) := \int_{\{\zeta>0\}} \left(\liminf_{x\to+\infty} f(t,x)-\mu_1 x \right) \zeta(t)\,dt \,+\\
\int_{\{\zeta<0\}} \left(\limsup_{x\to-\infty} f(t,x)-\nu_1 x \right) \zeta(t)\,dt \,,
\end{multline*}
\begin{multline*}
\mathcal A^+(\zeta) := \int_{\{\zeta>0\}} \left(\limsup_{x\to+\infty} f(t,x)-\mu_2 x \right) \zeta(t)\,dt \, +\\
\int_{\{\zeta<0\}} \left(\liminf_{x\to-\infty} f(t,x)-\nu_2 x \right) \zeta(t)\,dt \,.
\end{multline*}

Collecting all the possibile situations in a unique statement we can summarize the Landesman-Lazer Assumption~\ref{assLL} for the Dirichlet problem~\eqref{psc2} as follows.
\begin{assumption}\label{assLLD}
Assume ${\mathcal R} \subset B$ and %$\overline{\mathcal R} \cap \Sigma_D\neq \varnothing$. Then
\begin{eqnarray*}
\text{if } (\mu_1,\nu_1) \in C_{h,h+1} &\text{then}& \mathcal A^-(\psi_{\mu_1,\nu_1})>0\,,\\
\text{if } (\mu_1,\nu_1) \in C_{h+1,h} &\text{then}& \mathcal A^-(\phi_{\mu_1,\nu_1})>0\,,\\
\text{if } (\mu_1,\nu_1) \in C_{h,h} &\text{then}& \mathcal A^-(\phi_{\mu_1,\nu_1})>0 \text{ and } \mathcal A^-(\psi_{\mu_1,\nu_1})>0\,,\\
\text{if } (\mu_2,\nu_2) \in C_{k,k+1} &\text{then}& \mathcal A^+(\psi_{\mu_2,\nu_2})<0\,,\\
\text{if } (\mu_2,\nu_2) \in C_{k+1,k} &\text{then}& \mathcal A^+(\phi_{\mu_2,\nu_2})<0\,,\\
\text{if } (\mu_2,\nu_2) \in C_{k,k} &\text{then}& \mathcal A^+(\phi_{\mu_2,\nu_2})<0 \text{ and } \mathcal A^+(\psi_{\mu_2,\nu_2})<0\,.\\
\end{eqnarray*}
\end{assumption}
\begin{remark}
Concerning the previous assumption,
if $\mathcal R\subset B_{2j}$ for a positive integer $j$ then only the case $(\mu_1,\nu_1) \in C_{j,j}$ can hold, and similarly we can have $(\mu_2,\nu_2) \in C_{j,j+1}$ if $\mu\leq\nu$ or $(\mu_2,\nu_2) \in C_{j+1,j}$ if $\mu\geq\nu$.
Similarly, if $\mathcal R\subset B_{2j-1}$ for a positive integer $j$ then only the case $(\mu_2,\nu_2) \in C_{j,j}$ can hold, and similarly we can have $(\mu_1,\nu_1) \in C_{j,j-1}$ if $\mu\leq\nu$ or $(\mu_1,\nu_1) \in C_{j-1,j}$ if $\mu\geq\nu$.
\end{remark}

Let us conlcude this paper with the existence theorem for the Dirichlet problem.

\begin{theorem}\label{TD}
Consider the problem~\eqref{psc2} where the continuous function $f$ satisfies~\eqref{fm} and~\eqref{fp}.
%If $\overline{\mathcal R}\subset B$ then there exists a solution.
%If ${\mathcal R} \subset B$ and $\overline{\mathcal R} \cap \Sigma_D\neq \varnothing$ and

If Assumption~\ref{assLLD} holds, then there exists at least one solution of~\eqref{psc2}.
\end{theorem}

\section*{Acknowledgements}

The author thanks Alessandro Fonda and Maurizio Garrione for the useful suggestions and discussions.

\end{document}